\documentclass[letterpaper,english,11pt]{amsart}

\usepackage[centering]{geometry}
\usepackage[french,main=english]{babel}
\usepackage[utf8, latin1]{inputenc} 
\usepackage[T1]{fontenc}
\usepackage[cm]{aeguill}   
\usepackage{times,newtxmath,ae,euscript,enumerate,dsfont}

\newcommand{\id}{\mathcal{I}(M)}
\newcommand{\st}{\mathcal{S}(M)}

\newcommand{\HK}{\mathcal{HK}}
\newcommand{\K}{\mathcal{K}}
\newcommand{\h}{\mathcal{H}}

\usepackage[dvipsnames]{xcolor}   
\usepackage{xparse}
\usepackage{xr-hyper}
\usepackage[linktocpage=true,colorlinks=true,hyperindex,citecolor=RoyalBlue,linkcolor= RoyalBlue]{hyperref}   

\theoremstyle{plain}
\newtheorem{theorem}{Theorem}[section] 
\newtheorem{proposition}[theorem]{Proposition}
\newtheorem{lemma}[theorem]{Lemma}
\newtheorem{corollary}[theorem]{Corollary}
\theoremstyle{definition}

\numberwithin{equation}{section}

\begin{document}
\author{A. Goswami}
\address{
[1] Department of Mathematics and Applied Mathematics, University of Johannesburg, P.O. Box 524, Auckland Park 2006, South Africa. [2] National Institute for Theoretical and Computational Sciences (NITheCS), South Africa.}
\email{agoswami@uj.ac.za}

\title{Terminal spaces of monoids}

\subjclass{20M12, 20M14, 54F65}
%Ideal theory for semigroups

%Commutative semigroups

%Topological characterizations of particular spaces

\keywords{strongly irreducible ideals; arithmetic monoids; Zariski topology; generic points}

\begin{abstract}
The purpose of this note is a wide generalization of the topological results of various classes of ideals of rings, semirings, and modules, endowed with Zariski topologies, to strongly irreducible ideals (endowed with Zariski topologies) of monoids, called terminal spaces. We show that terminal spaces are $T_0$, quasi-compact, and every nonempty irreducible closed subset has a unique generic point. We characterize arithmetic monoids in terms of terminal spaces. Finally, we provide necessary and sufficient conditions for the subspaces of maximal and prime ideals to be dense in the corresponding terminal spaces.
\end{abstract}
\maketitle

\section{Introduction and Preliminaries}

Under the name primitive ideals, in \cite{F49}, the notion of strongly irreducible ideals was introduced for commutative rings. In \cite[p.\,301, Exercise 34]{B72}, the ideals of the same spectrum are called quasi-prime ideals. The term ``strongly irreducible'' was first used for noncommutative rings in \cite{B53}. Since then, several algebraic and topological studies have been done on these types of ideals of rings (see \cite{HRR02, A08, S16}). The notion of strongly irreducible ideals has been generalized to semirings (see \cite{I56, AA08}) and modules (see \cite{KES06, NS21}).

The aim of this note is to study the topological properties of the space of strongly irreducible ideals of a monoid endowed with a Zariski topology. This is a wide generalization of Zariski spaces. Moreover, strongly irreducible ideals are the ``largest'' class of ideals on which one can impose a Zariski topology. Therefore, we not only generalize some of the topological results from the above-mentioned works on strongly irreducible ideals of rings, semirings, and semimodules to monoids but also from maximal, prime, minimal prime, and primary ideals of those structures to strongly irreducible ideals of monoids. We highlight the results that have been generalized here. Although our setup is on monoids, many of the results still hold for (commutative) semigroups, which cannot be further generalized to magmas.

By a monoid $M $, we shall mean a system $(M , \cdot,1)$ such that $\cdot$ is an associative, commutative  (multiplicative) binary operation on $M $ and $1$ is an element of $M $ such that $x\cdot 1=x$ for all $x\in M $. We shall write $xy$ for $x\cdot y$ and we shall assume all our monoids are commutative. An element $m$ of $M$ is called \emph{invertible} if $mm'=1,$ for some $m'\in M.$ If $S$ and $T$ are subsets of a monoid $M$, then by the \emph{set product} $ST$ of $S$ and $T$ we shall mean $ST=\{st\mid s\in S, t\in T\}.$ If $S=\{s\}$ we write $ST$ as $sT$, and similarly for $T=\{t\}.$ Thus
\[ST=\cup\{ St\mid t\in T\}=\cup \{sT\mid s\in S\}.\]

An \emph{ideal}  of a monoid $M $ is a subset $I$ such that $i\in I$ and
$m\in M $ implies $im \in I$, and $I$ is called a \emph{proper} ideal if $I\neq M $. By $\id$, we shall denote the set of all ideals of $M$. A monoid $M $
is called \emph{Noetherian} if it satisfies the
ascending chain condition on ideals. If $S$ is a nonempty subset of a monoid $M$, then $\langle S\rangle$ will denote the smallest ideal generated by $S$. If $S=\{s\},$ then we shall write $\langle s \rangle$ for $\langle \{s\} \rangle$. A proper ideal I is called \emph{prime} if $i i' \in I$
implies $i\in I$ or $i'\in I$. If $I$ and $J$ are ideals of $M $, then their \emph{product} is defined by $IJ=\{ij\mid i\in I, j\in J\}$, which is also an ideal of $M $. Let $S$ be a nonempty subset of a monoid $M$.  If $I$ is an ideal of $M $, the \emph{radical} of $I$ is defined by \[\sqrt{I}=\{m\in M \mid m^k\in I\;\text{for some}\;k\in \mathds{Z}^+\}.\] An ideal $I$ is said to be a \emph{radical ideal} (or to be a \emph{semiprime}) if $\sqrt{I}=I.$  A proper  ideal $L$ of $M $ is called \emph{irreducible} if $L=I\cap J$ (for ideals $I, J$ of $M$) implies $L=I$ or $L=J.$ A proper ideal $K$ of $M$ is said to be \emph{strongly irreducible} if $I\cap J\subseteq K$ implies $I\subseteq K$ or $J\subseteq K$. 

\section{Terminal spaces}

Let $M$ be a monoid and let $\st$ be the set of all strongly irreducible ideals of $M$. We impose a Zariski topology (in the sense of \cite[\textsection 1.1.1]{G60}) on $\st$ by defining closed sets by
\begin{equation}\label{kcp}
\HK(X)=\begin{cases}
\{J\in \st \mid J\supseteq \K(X)\},& X\neq \emptyset;
\\ \emptyset,& X=\emptyset,
\end{cases}
\end{equation}
where $X\subseteq \st$ and $\K(X)=\bigcap_{I\in X} I$.
The following theorem shows that  $\HK$ is a Kuratowski closure operator on $\st$, and hence indeed induces a closed-set topology on $\st$. 

\begin{theorem}\label{klo}
Let $M$ be a module and let  $\HK$ be defined as in $($\ref{kcp}$)$.
\begin{enumerate}
	
\item $\HK(\emptyset)=\emptyset$.

\item\label{xhkx} For all $X\subseteq \st,$ $X\subseteq \HK(X).$

\item For all $X\subseteq \st,$ $\HK(\HK(X))=\HK(X)$.

\item\label{vss} For all $X, X'\subseteq \st,$ $\HK(X\cup X')=\HK(X)\cup \HK(X').$
\end{enumerate}
\end{theorem}

\begin{proof}
(1)--(2) Follows from (\ref{kcp}).

(3) By (2), $X\subseteq \HK(X)$ and hence $\HK(\HK(X))\supseteq \HK(X),$ by increasing property of $\HK$. The other inclusion follows from (\ref{kcp}).

(4) By (2) and by the increasing property of $\HK$, we have $\HK(X\cup X')\supseteq \HK(X)\cup \HK(X').$ Suppose $J\in \HK(X\cup X')$. Then $\K (X) \cap \K(X')\subseteq J.$
Since $J$ is strongly irreducible, $\K (X) \subseteq J$ or $\K (X') \subseteq J$, and hence, $J\in \HK(X)\cup \HK(X').$ 
\end{proof}

From Theorem \ref{klo}(\ref{vss}), it is clear that the class of strongly irreducible ideals is the ``largest'' class of ideals of a monoid on which we can endow a hull-kernel topology (= Zariski topology). The set $\st$ endowed with the above-mentioned hull-kernel topology will be called a \emph{terminal space}. The following proposition characterizes strongly irreducible ideals as terminal spaces, and it generalizes the ring-theoretic result \cite[Section 2.2, p.\,11]{M53}.

\begin{proposition}
The operation defined in $($\ref{kcp}$)$ is a Kuratowski closure operator on a class   $\mathcal{F}$ of ideals of $M$  if and only if   
\[
J\cap K\subseteq I\quad \text{implies}\quad  J\subseteq I\;\; \text{or}\;\; K\subseteq I, 
\] for all $J, K\in \id$ and for all $I\in  \mathcal{F}.$ 
\end{proposition}

Before we discuss topological properties of terminal spaces, let us note down a few more elementary results about the closure operator $\HK$, which will be used in the sequel.

\begin{lemma}\label{fphk}
Let $M$ be a monoid and let $X$, $X'$, $\{X_{\lambda}\}_{\lambda \in \Lambda}$ be nonempty subsets of $\st$. Then the following hold.

\begin{enumerate}
\item $\HK(M)=\emptyset.$	
	
\item\label{hkxb} $\HK(X)=\overline{X}.$

\item\label{hku} $\HK(X)\cup \HK(X')=\HK(X\cap X').$

\item\label{hkx} $\bigcap_{\lambda \in \Lambda} \HK(X_{\lambda})=\HK\left( \bigcap_{\lambda \in \Lambda} X_{\lambda} \right).$

\item $\HK(X)\subseteq \HK(\langle X\rangle)\subseteq \HK(\sqrt{\langle X\rangle}).$
\end{enumerate}
\end{lemma}

\begin{proof}
(1) Follows from the definition of a strongly irreducible ideal of $M$.

(2) From Theorem \ref{klo}(\ref{xhkx}), we have $\overline{X}\subseteq \overline{\HK(X)}=\HK(X).$ Let $\HK(Y)$ be an arbitrary closed subset of $\st$ containing $X$. Then \[\HK(Y)=\HK(\HK(Y))\supseteq \HK(X).\]
Since by Theorem \ref{klo}(\ref{xhkx}), $\HK(X)$ is the smallest closed set containing $X$, we have the claim.

(3)--(5)  Straightforward.
\end{proof}

The next result generalizes Theorem 4.1 and Theorem 3.1 in \cite{HPH21}, Theorem 9 in \cite{I56}, Theorem 4.1:(v)--(vi) in \cite{A08},  and Proposition 2.4 in \cite{Z00}.

\begin{theorem}\label{cat0}
Every terminal space is quasi-compact and a $T_0$-space.
\end{theorem}

\begin{proof}
Let  $\{C_{ \lambda}\}_{\lambda \in \Lambda}$ be a family of closed sets of $\st$  and let  $\bigcap_{\lambda\in \Lambda}C_{ \lambda}=\emptyset.$ Then $C_{ \lambda}=\HK(X_{\lambda})$ for some subsets $X_{\lambda}$ of $\st$, and by Lemma \ref{fphk}(\ref{hkx}), we have \[\bigcap_{\lambda \in \Lambda} \HK(X_{\lambda})=\HK\left( \bigcap_{\lambda \in \Lambda} X_{\lambda} \right)=\emptyset.\] Let
$K=\left\langle \bigcup_{\lambda \in \Lambda} \K (X_{\lambda}) \right\rangle.$ We claim that $K=M$.
If not, then there exists a maximal ideal $J$ of $M$ such that
\[\bigcap_{I\in X_{\lambda}} I\subseteq K\subseteq J,\]
for all $\lambda \in \Lambda.$ Therefore, $J\in \h(C_{\lambda})=C_{\lambda}$ for all $\lambda \in \Lambda$, a contradiction. Since $1\in K,$ we have  $1\in \bigcup_{i=1}^n \K (X_{\lambda_i}) ,$ for a finite subset $\{\lambda_{\scriptscriptstyle 1}, \ldots, \lambda_{\scriptscriptstyle n}\}$ of $\Lambda$. Hence, $\bigcap_{i=1}^nC_{\lambda_i}=\emptyset,$ and by the finite intersection property, we have the quasi-compactness of $\st$. 

To show the $T_0$ separation property, let $I, I'\in \st$ such that $\HK(\{I\})=\HK(\{I'\}).$ It suffices to show $I=I'$. Since $I'\in \HK(\{I\}),$ we have $I\subseteq I'$. Similarly, we obtain $I'\subseteq I$. Hence $I=I'$.  	
\end{proof}

The following result characterizes  $T_1$ terminal spaces, and generalizes Theorem 3.2 in \cite{HPH21}, Theorem 3.7 in \cite{HHP22}, and Theorem 3 in \cite{Y11}.

\begin{theorem}
Let $M$ be a monoid. A terminal space $\st$ is a $T_1$-space if and only if every strongly irreducible ideal of $M$	is not containing in the other strongly irreducible ideals of $M$.
\end{theorem}

\begin{proof}
If $\st$ is a $T_1$-space, then for every $I\in \st$ we have $\bar{I}=\{I\}.$ By Lemma \ref{fphk}(\ref{hkxb}), $\bar{I}=\HK(\{I\})=\mathcal{H}(I),$ and so, $\{I\}=\mathcal{H}(I),$ implying that the only strongly irreducible ideal of $M$ containing $I$ is $I$ itself. For the converse, let $I$ be the unique strongly irreducible ideal of $M$ that contains $I$. Then by Lemma \ref{fphk}(\ref{hkxb}),
\[\overline{\{I\}}=\HK(\{I\})=\mathcal{H}(I)=\{I\}.\]
Thus $\{I\}$ is a closed set, proving that $\st$ is a $T_1$-space.
\end{proof}

Our next goal is to study generic points of irreducible closed sets of terminal spaces. Recall that a subset $Y$ of a topological
space $X$ is called \emph{irreducible} if for any closed subsets $Y_1$ and $Y_2$ in $X,$ $Y \subseteq  Y_1 \cup Y_2$ implies that
$Y \subseteq  Y_1$ or $Y \subseteq Y_2$. A maximal irreducible subset $Y$ of $X$ is called an \emph{irreducible component}. An element $y$ of a closed subset $Y$ of $X$ is called a \emph{generic point of} $Y$ if $Y=\overline{\{y\}}$. The following result characterizes irreducible subsets of a terminal space. Moreover, this result generalizes Theorem 3.3 in \cite{HPH21}, Proposition 3 in \cite{Y11}, and Theorem 2.6(1) in \cite{Z00}.

\begin{theorem}\label{cis}
Let $M$ be a monoid. A nonempty closed subset $X$ of a terminal space $\st$ is irreducible if and only if $\K(X)$ is a strongly irreducible ideal of $M$.
\end{theorem}

\begin{proof}
It is clear that $\K(X)$ is a proper ideal of $M$. Let $I\cap J\subseteq \K(X)$ for some $I, J\in \id$. Then for any $L\in X$, we have $I\subseteq L$ or $J\subseteq L$ since $L\in \st$. Hence $X\subseteq \h(I)\cup \h(J).$ Since $X$ is irreducible, $X\subseteq \h(I)$ or $X\subseteq \h(J),$ which implies that $I\subseteq \K(X)$ or $J\subseteq \K(X).$ Therefore, $\K(X)$ is a strongly irreducible.

For the converse, let $\K(X)$ be a strongly irreducible ideal fo $M$. Since $\K(X)\neq M$, $\K(X)$ is nonempty. Let $X=X_1\cup X_2$ for some nonempty closed subsets of the terminal space $\st$. Then $\K(X) \supseteq \K(X_1)\cap \K(X_2).$ Since $\K(X)$ is strongly irreducible, $\K(X)\in \h(\K(X_1)\cap \K(X_2)).$ By Lemma \ref{fphk}(\ref{hku}), this implies $\K(X)\in \HK(X_1)\cup \HK(X_2).$ If $\K(X)\in \HK(X_1),$ then 
\[X\subseteq \overline{X}=\HK(X)\subseteq \HK(X_1)=\overline{X_1}=X_1,\]
where the first and the second equalities follow from Lemma \ref{fphk}(\ref{hkxb}). Similarly, if $\K(X)\in \HK(X_2),$ then obtain $X\subseteq X_2.$ This proves that $X$ is irreducible.
\end{proof}

The following corollary generalizes Corollary 3.1 in \cite{HPH21}.

\begin{corollary}\label{ugp}
Every nonempty irreducible closed subset of a terminal space $\st$ has a unique generic point.
\end{corollary}

\begin{proof}
Let $\h(I)$ be a nonempty irreducible subset of $\st$. Then by Theorem \ref{cis}, $I$ is strongly irreducible. Hence $\overline{\{I\}}=\HK(I)=\h(I),$ where the first equality follows from Lemma \ref{fphk}(\ref{hkxb}). Thus $\{I\}$ is a generic point of $\h(I).$ The uniqueness of this point follows from the fact that $\st$ is a $T_0$-space (see Theorem \ref{cat0}). 
\end{proof}

The following one-to-one correspondence generalizes Theorem 3.4 in \cite{AH11}.

\begin{theorem}
Let $M$ be a monoid. Then there is a bijection between the set of irreducible components of the terminal space $\st$ and the set of minimal strongly irreducible ideals of $M$.
\end{theorem}

\begin{proof}
If $X$ is an irreducible component of the terminal space $\st$, then by Theorem \ref{cis}, $X=\h(I)$ for some $I\in \st$. If $J\in \st$ such that $I\supseteq J,$ then $\h(I)\subseteq \h(J)$ so that $I=J$. Conversely, let $I$ be a minimal strongly irreducible ideal of of $M$ and let $\h(I)\subseteq \h(J)$ for some $J\in \st.$ Then 
\[\overline{\{I\}}=\h(I)\subseteq \h(J)=\overline{\{J\}},\]
implying that $I=J$. Hence, $\h(I)$ is an irreducible component of $\st$.
\end{proof}

A characterization of invertible elements of $\st$ is given in the following proposition and this result generalizes Theorem 4.1(iii) in \cite{A08}.

\begin{proposition}
Let $M$ be a monoid and $\st$ be a terminal space. Then $\st\setminus \HK(\langle m \rangle)=\st$ if and only if $m$ is an invertible element of $M$.
\end{proposition}

\begin{proof}
Note that $\st\setminus \HK(\langle m \rangle)=\st$ implies $m$ is not in any maximal ideal of $M$, and hence $m$ is invertible. The converse follows immediately from the fact that every strongly irreducible ideal is proper.
\end{proof}

It is well-known that prime spectrum of a Noetherian (commutative) ring endowed with Zariski topology is a Noetherian space. The following proposition generalizes this to strongly irreducible ideals of monoids, and it also generalizes Proposition 4.2(i)  in \cite{A08}. The  proof is easy, and so will be omitted.

\begin{proposition}
If $M$ is a Noetherian monoid, then $\st$ is a Noetherian terminal space.
\end{proposition}

Recall that a monoid is called \emph{arithmetic} if $\id$ is a distributive lattice. The following theorem characterizes arithmetic monoids in terms of strongly irreducible ideals. This result is a generalization of Theorem 10 in \cite{I56}. The half of the implications uses the Zariski topology on $\st$.

\begin{theorem}
%[I56], Th. 10
A monoid $M$ is arithmetic if and
only if each ideal is the intersection of all strongly irreducible ideals containing it.
\end{theorem}

\begin{proof}
Let $I\in \id$ and let $I=\bigcap_{I\subseteq J}\{J\mid J\in \st\}.$ To show $\id$ is distributive, it suffices to show that the lattice $\id$ is isomorphic to the lattice of some closed sets of the terminal space $\st$. Note that the map $I\mapsto \{ J\in \st \mid J\supseteq I\}=\h(I)$ is a bijection and since $\h(I)$ is a closed set, this map is also an lattice isomorphism.

For the converse,  we first observe that by \cite{BF48}, in a distributive lattice, irreducible ideals and strongly irreducible ideals coincide. The rest of the proof now follows from Theorem 6 and Theorem 7 in \cite{I56}.
\end{proof}

Finally, we wish to see relations between  a terminal space and its subspaces of maximal ideals $\mathrm{Max}(M)$ and  prime ideals $\mathrm{Spec}(M)$. To do so, we first talk about radicals induced by maximal, prime, and strongly irreducible ideals of a monoid $M$. An \emph{m-radical} $\sqrt[m]{M}$ (respectively, \emph{p-radical} $\sqrt[p]{M}$ and \emph{s-radical} $\sqrt[s]{M}$) of $M$ is the intersection of all maximal ideals (respectively, prime ideals and strongly irreducible ideals) of $M$.

\begin{proposition}
Let $M$ be a monoid.
\begin{enumerate}
	
\item The subspace $\mathrm{Max}(M)$ is dense in the terminal space $\st$ if and only if $\sqrt[p]{M}=\sqrt[s]{M}.$

\item The subspace $\mathrm{Spec}(M)$  is dense in the terminal space $\st$ if and only if $\sqrt[m]{M}=\sqrt[s]{M}.$
\end{enumerate}
\end{proposition}

\begin{proof}
(1) Let $\overline{\mathrm{Spec}(M)}=\st.$ Then $\{J\in \st\mid \bigcap_{P\in \mathrm{Spec}(M)}P\subseteq J\}=\st.$ This implies
\[\sqrt[p]{M}=\bigcap_{P\in \mathrm{Spec}(M)}P\subseteq \bigcap_{J\in \st}J=\sqrt[s]{M}.\]
Furthermore, $\mathrm{Max}(M)\subseteq \st$ implies $\sqrt[s]{M}\subseteq \sqrt[p]{M}$. Hence, we have the desired equality. To obtain the converse, let $\st\setminus \overline{\mathrm{Spec}(M)}\neq \emptyset.$ This implies $J\notin \overline{\mathrm{Spec}(M)}$, but $J\in \st.$ Therefore, there exists a neighbourhood $N_J$ of $J$ such that $N_J\cap \mathrm{Spec}(M)=\emptyset,$ and $\sqrt[s]{M}\subsetneq \sqrt[p]{M}.$ In other words, we have  $\sqrt[s]{M}\neq \sqrt[p]{M}.$

(2) Follows from (1).
\end{proof}

\end{document}